\documentclass[a4paper,12pt]{article}
\usepackage{amsmath}
\usepackage{amsfonts}
\usepackage{amssymb}
\usepackage{amsthm}
\usepackage{graphicx}
\usepackage{ifthen}
\usepackage{verbatim}
\usepackage{amsmath}
\usepackage{setspace}
\usepackage{amsfonts}
\usepackage{amssymb}
\usepackage{amsthm}
\usepackage{lscape}
\usepackage{setspace}
\usepackage{listings}
\usepackage{amsmath}
\usepackage{amsfonts}
\usepackage{amssymb}
\usepackage{amsthm}
\usepackage{ifthen}
\usepackage{verbatim}
\usepackage{xcolor}
\newtheorem{thm}{Theorem}[section]
\newtheorem{ex}[thm]{Example}

\newtheorem{defn}[thm]{Definition}
\newtheorem{rem}[thm]{Remark}
\newtheorem{prop}[thm]{Proposition}

\newcommand{\cadlag}{{c\`adl\`ag }}
\def\1{\mathbf{1}}
\def\R{\mathbb{R}}
\def\P{\mathbb{P}}

\def\E{\mathbb{E}}
\def\N{\mathbb{N}}

\def\d{\,\mathrm{d}}
\def\dd{\mathrm{d}}

\def\G{\Gamma}

\def\F{\mathcal{F}}
\def\G{\mathcal{G}}

\def\T{\mathbb{T}}

\def\D{\Delta}

\begin{document}
\title{Captive Jump Processes}
\author{Andrea Macrina$^{\dag}$, Levent A. Meng\"ut\"urk{$^{\dag}$}, Murat C. Meng\"ut\"urk{$^{\ddag}$\thanks{ a.macrina@ucl.ac.uk, l.menguturk@ucl.ac.uk, murat.menguturk@ozyegin.edu.tr}} \\ \\ {$^{\dag}$Department of Mathematics, University College London} \\ {London, United Kingdom} 
\\ {$^{\ddag}$}Center for Financial Engineering, Ozyegin University \\ Istanbul, Turkey}
\date{\today}
\maketitle

\begin{abstract}
We explicitly construct so-called captive jump processes. These are stochastic processes in continuous time, whose dynamics are confined by a time-inhomogeneous bounded domain. The drift and volatility of the captive processes depend on the domain's boundaries which in turn restricts the state space of the process. We show how an insurmountable shell may contain the captive jump process while the process may jump within the restricted domain. In a further development, we also show how---within a confined domain---inner time-dependent corridors can be introduced, which a captive jump process may leave only if the jumps reach far enough, while nonetheless being unable to ever penetrate the outer confining shell. Captive jump processes generalize the recently developed captive diffusion processes. In the case where a captive jump-diffusion is a continuous martingale or a pure-jump process, the uppermost confining boundary is non-decreasing, and the lowermost confining boundary is non-increasing. Captive jump processes could be considered to model phenomena such as electrons transitioning from one orbit (valence shell) to another, and quantum tunneling where stochastic wave-functions can ``penetrate'' boundaries (i.e., walls) of potential energy. We provide concrete, worked-out examples and simulations of the dynamics of captive jump processes in a geometry  with circular boundaries, for demonstration.
\end{abstract} 
{\bf MSC2020:} 60

\section{Introduction}
In this work, we develop \emph{captive jump processes}, which may include jump-diffusion, pure jump processes, and path-dependent processes with jumps, all of which have a confined path space. These stochastic processes generalize the recently introduced captive diffusion processes by Meng\"ut\"urk \& Meng\"ut\"urk (2020) along with several potential applications thereof, see Meng\"ut\"urk \& Meng\"ut\"urk (2021a) for an example in quadratic optimization, and Meng\"ut\"urk \& Meng\"ut\"urk (2021b) to model random particle systems.

In many physical and social systems, one encounters random processes that are \emph{restricted} in their dynamics, where the stochastic phenomenon stays within a given topological subspace. We emphasize several examples such as Skorokhod-type stochastic differential equations (SDEs) (\cite{1,2,4,30}), diffusion processes on submanifolds (\cite{1aa,1ab,27a}), reflected diffusions (\cite{9,14,30ab}), Brownian excursions (\cite{6,28,33,41}), non-colliding diffusions (\cite{3,29,44,47}), Bessel processes (\cite{8,11,25, 48ab}) and captive diffusions \cite{50, 51, 52}.

The stochastic processes mentioned above, one common property they share is that they all have continuous paths. However, there are many observed systems, which display discontinuities (at random times), possibly due to shocks producing substantial impact on the evolution of the system. Here we mention transitions of quantum systems across excitation levels, e.g., electrons jumping to a different energy level, sudden price reactions to news in financial markets, and so on. Thus, studying processes manifesting pure-jump or jump-diffusion dynamics attracts a great deal of interest in numerous applications and modelling purposes---if not for their sake. 

Captive jump processes, by construction, cannot leave a given bounded domain, we call the \emph{master domain}. Within this domain, these processes may have either pure-jump, continuous, or jump-continuous dynamics---in fact, such a process may display all of these dynamics over non-overlapping time intervals. To construct captive jump processes, we rely on the approach developed in \cite{50}, and thus generalize it at the same time. Pure-jump and jump-diffusion processes are controlled by coefficients of stochastic differential equations (SDE), which satisfy certain regularity conditions with respect to a pair of \cadlag paths acting as the boundaries. One may further sub-divide the bounded domain into subspaces (e.g., internal corridors as in \cite{52}) and require by construction that the process stay within one of these restricting subspaces. At the same time, however, we can allow for non-zero probabilities that the processes may jump from one subspace to another, while still staying constricted within the master domain, almost surely.

The structure of this paper is as follows. Section 2 introduces captive jump-diffusion processes and studies some of their properties. In Section 3, we apply the framework to the physical phenomena as mentioned above, along with some simulations for demonstration purposes. 
\section{Construction and properties of captive jump processes}
We consider a filtered probability space $(\Omega,\F,(\F_{t})_{t \leq \infty},\P)$, $\F_{\infty}=\F$, where all filtrations are right-continuous and complete. We denote the Borel $\sigma$-field by $\mathcal{B}$ and introduce the time interval $\T=[0,T]$ with some fixed horizon $T<\infty$. We let $\mathcal{D}(\R)\subset\Omega$ be the Skorokhod space of \cadlag functions, where $X:\mathcal{D}(\R)\rightarrow \R$ produces a \cadlag process $(X_t)_{t\in\T}$. By $(\mathcal{F}_t^{X})_{t\in\T}$, we denote the natural filtration of $(X_t)_{t\in\T}$ such that $\mathcal{F}_t^{X}\subset\mathcal{F}_t$ for any $t\in\T$.

We start next with the construction of Markovian captive jump processes, after which we generalize to a path-dependent setting.
\subsection{Capitve jump-diffusion processes}
The space of continuous $\R$-valued $\{(\F_t),\P\}$-Markov martingales (e.g., Brownian motion) is denoted $\mathcal{M}(\R)\subset\mathcal{C}(\R)$, where $\mathcal{C}(\R)\subset\mathcal{D}(\R)$ is the space of continuous functions. We denote $\mathcal{J}(\R)\subset\mathcal{D}(\R)$ as the space of discontinuous functions so that any element $(J_t)_{t\in\T}$ of $\mathcal{J}(\R)$ is a Markov jump process with finite activity and jump-size equal to one. That is, if $(J_t)_{t\in\T}$ is a stochastic process, its randomness arises only from its jump-times (e.g., the Poisson process). We write 
$\D J_t = J_t - J_{t-}$ for $t>t-$, for all $t\in\T$. So, if $\D J_t = 0$, then the process $(J_t)$ is continuous at $t$. If $\D J_t \neq 0$, then this implies there is a discontinuity at $t$. The following function space will model the boundaries that define the restricted domains.
\begin{defn}
Let $\widetilde{\G}\subset\mathcal{D}(\R)$ be a set of measurable c\`adl\`ag functions, where for any  $g\in\widetilde{\G}$, $g:\T\rightarrow\R$ is a locally bounded map with a locally bounded right-derivative $\dd g_+(t)/\dd t$ on the intervals where $g(t)$ is continuous. We also introduce the following subspaces:
\begin{enumerate}
\item $\G\subset\widetilde{\G}$ is the space of (purely) continuous boundaries.
\item $\widetilde{\G}^{(l)}\subset\widetilde{\G}$ is such that $g\in\widetilde{\G}^{(l)}$ implies $\D g(t) \leq 0$ for all $t\in\T$.
\item $\widetilde{\G}^{(u)}\subset\widetilde{\G}$ is such that $g\in\widetilde{\G}^{(u)}$ implies $\D g(t) \geq 0$ for all $t\in\T$.
\end{enumerate}
\end{defn}
We now introduce a family of processes that extends \cite{50} and forms the main focus of this paper. Thereafter, we study their main properties and provide some examples below.
\begin{defn}
\label{maindeftwo}
Let $g^l\in\widetilde{\G}^{(l)}$ and $g^u\in\widetilde{\G}^{(u)}$ such that $g^l(t) < g^u(t)$ for all $t\in\T$. Then, a \emph{captive jump-diffusion process} $(X_t)_{t\in\T}\in\mathcal{D}(\R)$ is the solution to an SDE governed by
\begin{align}
\label{mainsdeonemain}
X_t = x_0 + \int_0^t \mu\left(s, X_s; g^l(s), g^u(s) \right)\dd s &+ \int_0^t \sigma\left(s, X_s; g^l(s), g^u(s)\right)\dd M_s \nonumber\\
&+ \sum_{0\leq s \leq t} \gamma\left(s-, X_{s-}; g^l(s-), g^u(s-)\right)\D J_s,
\end{align}
where $X_0 = x_0 \in [g^l(0), g^u(0)]$. The maps $\mu:\T\times\R \rightarrow \R$ and $\sigma:\T\times\R \rightarrow \R$ are locally bounded and continuous, and $\gamma:\T\times\R \rightarrow \R$ is locally bounded and \cadlag. They satisfy:
\begin{enumerate}
\item $\mu\left(t, g^l(t-); g^l(t), g^u(t)\right)\geq \dd g^l_+(t)/\dd t + \D g^l(t)$ and \\ $\mu\left(t, g^u(t-); g^l(t), g^u(t)\right)\leq \dd g^u_+(t)/\dd t + \D g^u(t)$,
\item $\sigma\left(t,g^l(t-); g^l(t), g^u(t) \right) = 0$ and $\sigma\left(t,g^u(t-); g^l(t), g^u(t) \right) = 0$,
\item $g^l(t-)-X_{t-} \leq \gamma\left(t-, X_{t-}; g^l(t-), g^u(t-)\right) \leq g^u(t-)-X_{t-}$,
\end{enumerate}
for all $t\in\T$, $\P$-a.s., assuming that $(M_t)_{t\in\T}\in\mathcal{M}(\R)$ and $(J_t)_{t\in\T}\in\mathcal{J}(\R)$ are mutually independent.
\end{defn}
To maintain a flexible level of generality, we \emph{define} captive jump-diffusion processes as solutions governed by Eq. (\ref{mainsdeonemain}) without imposing specific conditions on the coefficients of the SDE for existence and uniqueness. In order to get a unique (strong) solution, one can further require the coefficients satisfy Lipschitz continuity and linear growth.
\begin{prop}
\label{captivepropmain}
For any $t\in\T$, $g^l(t)\leq X_t\leq g^u(t)$ holds $\P$-almost surely.
\end{prop}
\begin{proof}
The continuous case is recovered if $\gamma\left(t, X_t; g^l(t), g^u(t)\right)=0$ for all $t\in\T$, which is proven in \cite{50}. We have $\D X_t \neq 0$ only when $\D J_t =1$, where the magnitude of $\D X_t$ is conditionally restricted for all $t\in\T$ by Property 3. in Definition \ref{maindeftwo}. Jumps can at most take $(X_t)_{t\in\T}$ onto a boundary, where the conditions on $\mu(\cdot)$ and $\sigma(\cdot)$ ensure reflection or absorption $\P$-a.s., given that $(J_t)_{t\in\T}$ has finite activity. Hence, $g^l(t)\leq X_t\leq g^u(t)$ for all $t\in\T$, $\P$-almost surely.
\end{proof}
The reason why we call these processes \emph{captive} stems from Proposition \ref{captivepropmain}---the process cannot break free from the restricted domain.
\begin{prop}
\label{sumofcaptives}
Let $(X^{(c)}_t)_{t\in\T}$ and $(X^{(d)}_t)_{t\in\T}$ be such that
\begin{enumerate}
\item $(X^{(c)}_t)_{t\in\T}$ is a captive continuous process with paths in $[g^l_c(t), g^u_c(t)]_{t\in\T}$, almost surely.
\item $(X^{(d)}_t)_{t\in\T}$ is a captive pure-jump process, and thus discontinuous, with paths in $[g^l_d(t), g^u_d(t)]_{t\in\T}$, almost surely.
\end{enumerate}
Then $(X_t)_{t\in\T}$ given by $X^{(c)}_t + X^{(d)}_t$ is a captive jump-diffusion process with path in $[g^l_c(t), g^u_c(t)]_{t\in\T}$, a.s., given that $(x^{(c)}_0 + x^{(d)}_0)\in[g^l_c(0),g^u_c(0))$, $g^l_c(t) \leq g^l_d(t)$ and $g^u_d(t) \leq g^u_c(t)$ for all $t\in\T$.
\end{prop}
\begin{proof}
If $\gamma=0$, then $(X_t)_{t\in\T}$ is a captive continuous process, see \cite{50}, and if $\mu = \sigma = 0$, then $(X_t)_{t\in\T}$ is a captive pure-jump process. The statement follows from the linear construction in Definition \ref{maindeftwo} and by using Proposition \ref{captivepropmain}.
\end{proof}

\begin{rem}
\label{diffusionrem}
Since neither $(M_t)_{t\in\T}$ nor $(J_t)_{t\in\T}$ are necessarily Markovian in general, captive jump-diffusion processes are not necessarily Markov processes. The term ``diffusion'' in this paper is to highlight continuity of paths and not Markovianity.
When $(X_t)_{t\in\T}$ is Markovian, we replace $(M_t)_{t\in\T}$ in (\ref{mainsdeonemain}) with an $\{(\F_t),\P\}$-Brownian motion $(W_t)_{t\in\T}$ for a canonical representation, and choose for $(J_t)_{t\in\T}$ an independent Markov jump process.
\end{rem}
\begin{ex}
\label{trigonometricexampleone}
As a trigonometric example, we first show that $(S_t)_{t\in\T}$ given by $S_t = \sin(W_t)$ is a captive-diffusion process between $(L^{(c)}_t)_{t\in\T}=-1$ and $(U^{(c)}_t)_{t\in\T}=1$. This follows since
\begin{align}
S_t &= -\frac{1}{2}\int_0^t\sin(W_u)\d u + \int_0^t\cos(W_u)\d W_u \notag \\
&= -\frac{1}{2}\int_0^t\sin(\sin^{-1}(S_u))\d u + \int_0^t\cos(\sin^{-1}(S_u))\d W_u  \notag \\
&= -\frac{1}{2}\int_0^tS_u\d u + \sqrt{1 - S_s^2}\d W_u, \notag
\end{align}
where $\mu\left(t, S_t; -1, 1 \right)=-\frac{1}{2}S_t$ satisfies Property 1. and $\sigma\left(t, S_t; -1, 1 \right)=\sqrt{1 - S_t^2}$ satisfies Property 2. of Definition \ref{maindeftwo}, respectively, at $(-1,1)$ for all $t\in\T$. Next, we can consider a captive pure-jump process $(Z_t)_{t\in\T}$ between $(L^{(d)}_t)_{t\in\T}=0$ and $(U^{(d)}_t)_{t\in\T}=1$ as follows:
\begin{align}
Z_t &= \sum_{0\leq s \leq t}\gamma_s \sqrt{1 - Z_s^2} \D J_s  \notag
\end{align}
where $\gamma_t\in(0,1]$ for all $t\in\T$. Hence, using Proposition \ref{sumofcaptives}, the following SDE provides a captive jump-diffusion process within $[-1,1]$ with positive jumps:
\begin{align}
X_t &= -\frac{1}{2}\int_0^t X_s\d s + \int_0^t\sqrt{1 - X_s^2}\d W_s + \sum_{0\leq s \leq t}\gamma_s \sqrt{1 - X_s^2} \D J_s.  \notag
\end{align}
One can replace $(S_t)_{t\in\T}$ above with $S_t = \cos(W_t)$, which is also a captive diffusion process.
\end{ex}
We can generalise the continuous part given in Example \ref{trigonometricexampleone} for a subclass of bounded functions. As such, the result below, together with Proposition \ref{sumofcaptives}, provides an alternative way to construct different captive jump-diffusion processes. 
\begin{prop}
\label{boundedfunctioncaptive}
Let $f:\R\rightarrow[a,b]$ be a bounded continuous function such that for $x=f(y)$,
\begin{enumerate}
\item $f$ has continuous inverse, i.e., $y=f^{-1}(x)$,
\item $f$ has bounded continuous first-derivative, i.e., $g(y) = \partial f/\partial y$ with $g(f^{-1}(a))=g(f^{-1}(b))=0$,
\item $f$ has  bounded continuous second-derivative, i.e., $h(y) = \partial^2 f/\partial y^2$ where $h(f^{-1}(a))\geq 0$ and $h(f^{-1}(b))\leq 0$.
\end{enumerate}
Then, $(X_t)_{t\in\T}$ given by $X_t = f(W_t)$ is a captive diffusion within $[a,b]$, for all $t\in\T$.
\end{prop}
\begin{proof}
Using It\^o's integration by parts formula and the conditions provided for the map $f$, we have
\begin{align}
X_t &= X_0 + \frac{1}{2}\int_0^t h(W_s)\d s + \int_0^t g(W_s)\d W_s, \notag \\
&= X_0 + \frac{1}{2}\int_0^t h(f^{-1}(X_s))\d s + \int_0^t g(f^{-1}(X_s))\d W_s. \notag
\end{align}
Since $f:\R\rightarrow[a,b]$, we have $f(0)=X_0\in[a,b]$. Having $h(f^{-1}(a))\geq 0$ and $h(f^{-1}(b))\leq 0$, Property 1., in Definition \ref{maindeftwo} is satisfied since $a$ and $b$ have vanishing right-derivatives. In addition, $g(f^{-1}(a))=g(f^{-1}(b))=0$ shows that Property 2., is satisfied. Since the process $(X_t)_{t\in\T}$ is continuous, $\gamma=0$, and the statement follows. 
\end{proof}
Resorting to Proposition \ref{boundedfunctioncaptive} is not the only way to construct captive jump-continuous processes. In fact, this construction only gives rise to a relatively small family of captive jump diffusions. We provide an alternative and more involved example below.
\begin{ex}
\label{meanrevertexamplejump}
Let $(J_t)_{t\in\T}$ be a Poisson process. The following process is a mean-reverting captive jump-diffusion with reflective boundaries:
\begin{align}
\label{meanrevertmin}
X_t = x_0 &+ \int_0^t \kappa_s(\beta_s - X_s ) \dd s + \int_0^t \alpha_s\left(X_s - L_s\right)\left(U_s - X_s\right)\dd W_s \notag \\
&+ \sum_{0\leq s \leq t} \theta_{s-} \min\left(X_{s-} - L_{s-}, U_{s-} - X_{s-}\right)\D J_s,  
\end{align}
for $x_0\in[L_0,U_0]$, where $(\alpha_t)_{t\in\T}$ and $(\kappa_t)_{t\in\T}>0$ are adapted continuous locally bounded maps, and $L_t < \beta_t < U_t$ for all $t\in\T$, so that $\kappa_t(\beta_t - L_t) > \dd L_+(t)/\dd t + \D L_t$ and $\kappa_t(\beta_t - U_t) < \dd U_+(t)/\dd t+ \D U_t$, for all $t\in\T$. Also, $(\theta_t)_{t\in\T}$ is a \cadlag map where $\theta_t\in[-1,1]/\{0\}$, for all $t\in\T$. 

In the figure below, we plot sample paths genereted by a simplified version of the SDE (\ref{meanrevertmin}), where we consider constant boundaries $(L_t)_{t\in\T}$ and $(U_t)_{t\in\T}$, and constant $(\beta_t)_{t\in\T}$.
\begin{figure}[h!]
\begin{center}
\includegraphics[width=8cm]{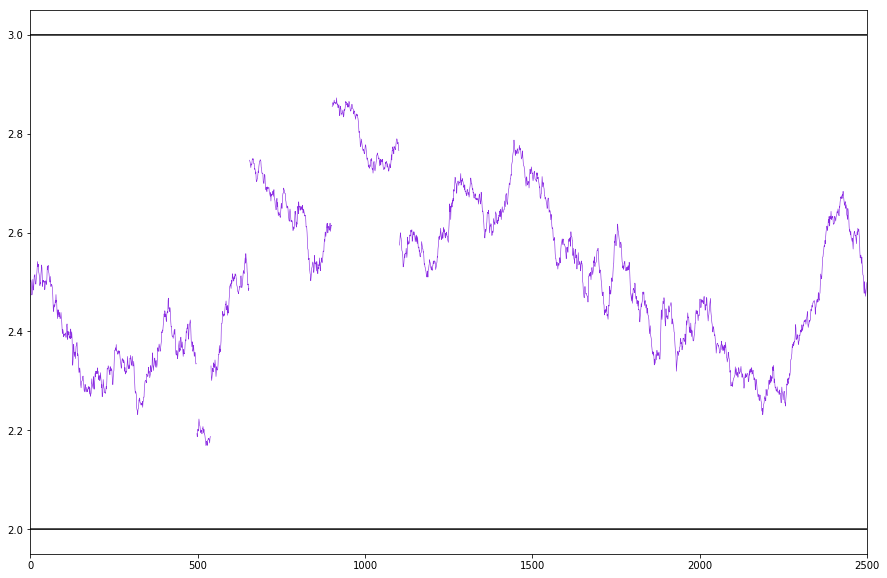}
\includegraphics[width=8cm]{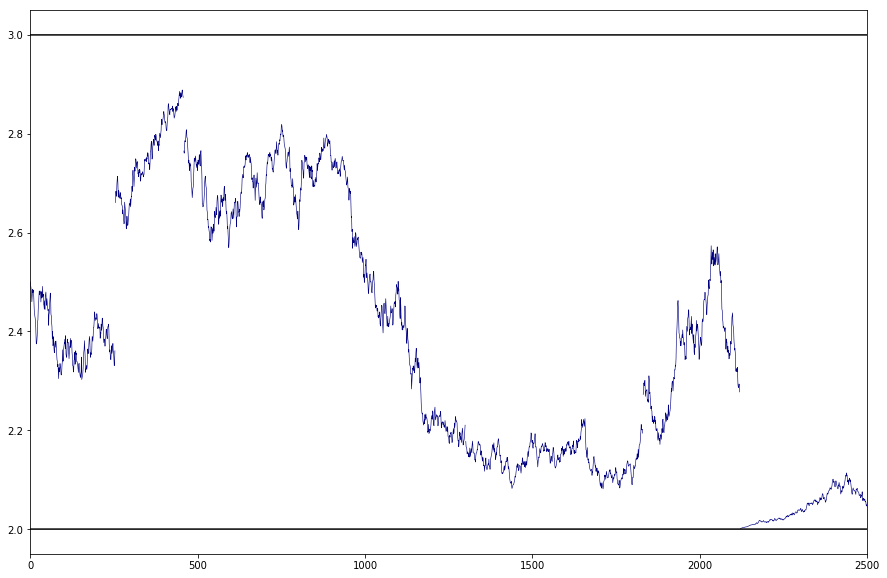} 
\caption[Captiveone]{Here, $(L_t)_{t\in\T}=2$ and $(U_t)_{t\in\T}=3$, and $(\beta_t)_{t\in\T}=2.5$.}
\end{center}
\end{figure}
\end{ex}
Since sums of semimartingales are semimartingales, if $(J_t)_{t\in\T}$ is a $\{(\F_t),\P\}$-semimartingale, then $(X_t)_{t\in\T}$ is a $\{(\F_t),\P\}$-semimartingale given that $\mu$ has locally bounded variation.
\begin{prop}
\label{monotononicbounds}
Let $L, U\in\mathcal{G}$. Then $(L_t)_{t\in\T}$ must be non-increasing and $(U_t)_{t\in\T}$ must be non-decreasing if either of the following holds:
\begin{enumerate}
\item $(X_t)_{t\in\T}$ is a continuous $\{(\F_t),\P\}$-martingale,
\item $(X_t)_{t\in\T}$ is a pure-jump process.
\end{enumerate}
\end{prop}
\begin{proof}
Since $L,U\in\mathcal{G}$, we have $\D L_t=\D U_t=0$ for any $t\in\T$. For the first part, if $(X_t)_{t\in\T}$ is a continuous $\{(\F_t),\P\}$-martingale then $\mu=\gamma=0$. Since $\mu\left(t, L_{t-}; L_t, U_t\right) = 0 \geq \dd L_+(t)/\dd t$ and $\mu\left(t, U_{t-}; L_t, U_t\right) = 0 \leq \dd U_+(t)/\dd t$ must hold for all $t\in\T$, it follows that $(L_t)_{t\in\T}$ must be non-increasing and $(U_t)_{t\in\T}$ must be non-decreasing. For the second part, if $(X_t)_{t\in\T}$ is a pure-jump process, then $\mu=\sigma=0$. Since the jump-times of $(J_t)_{t\in\T}$ do not depend neither on $(L_t)_{t\in\T}$ or on $(U_t)_{t\in\T}$, and since $\gamma$ cannot act on $(X_t)_{t\in\T}$ over time horizons where $(J_t)_{t\in\T}$ does not jump, we have the following:
\begin{align}
&\P(X_t > U_t)>0 \hspace{0.15in} \text{if $\dd U_+(t)/\dd t < 0$}, \notag \\ 
&\P(X_t < L_t)>0 \hspace{0.15in} \text{if $\dd L_+(t)/\dd t > 0$},  \notag 
\end{align}
for $t\in\T$. However, the captivity property is preserved if $\dd U_+(t)/\dd t \geq 0$ and $\dd L_+(t)/\dd t \leq 0$, since the boundaries cannot block the paths of $(X_t)_{t\in\T}$ even over time horizons where no jumps occur, and hence over periods during which $(X_t)_{t\in\T}$ remains constant.
\end{proof}
\begin{prop}
\label{conditionsonbounds}
Let $(L_t)_{t\in\T}$ and $(U_t)_{t\in\T}$ be constant and $\E^{\P}[\D J_t \,|\, \F_{t-} ]>0$ for all $t\in\T$. Then the boundaries must be absorbing if 
\begin{enumerate}
\item $(X_t)_{t\in\T}$ is a continuous $\{(\F_t),\P\}$-martingale,
\item $(X_t)_{t\in\T}$ is a pure-jump $\{(\F_t),\P\}$-martingale.
\end{enumerate}
\end{prop}
\begin{proof}
Let $\tau = \inf(t\geq 0 : X_t = L_t = L \,\cup\, X_t = U_t = U)$ be the first-hitting time to either of the boundaries where $\inf\emptyset=\infty$, and define $(Y_t)_{t\in\T}$ such that $Y_t = X_{t\wedge \tau}$ for all $t\in\T$. Since $(L_t)_{t\in\T}$ and $(U_t)_{t\in\T}$ are constant, we have $\dd L_+(t)/\dd t = \dd U_+(t)/\dd t = 0$. If $(X_t)_{t\in\T}$ is an $\{(\F_t),\P\}$-martingale, then we have $\E^{\P}[|X_t|] < \infty$ and
\begin{align}
\label{nodriftnojump}
\E^{\P}[X_t  \,|\, \F_u ] &= x_0 + \int_0^u \mu\left(s, X_s; L_s, U_s \right)\dd s + \int_0^u \sigma\left(s, X_s; L_s, U_s\right)\dd M_s + \sum_{0\leq s \leq u} \gamma\left(s-, X_{s-}; L_{s-}, U_{s-}\right)\D J_s \notag \\
&\hspace{.95cm}+ \E^{\P}\left[\int_u^t \mu\left(s, X_s; L_s, U_s \right)\,|\, \F_u \right]\dd s + \E^{\P}\left[\int_u^t \sigma\left(s, X_s; L_s, U_s\right)\dd M_s\,|\, \F_u \right] \notag \\
&\hspace{.95cm}+ \E^{\P}\left[\sum_{u < s \leq t} \gamma\left(s-, X_{s-}; L_{s-}, U_{s-}\right)\D J_s\,|\, \F_u \right] \notag \\
&= X_u + \int_u^t \E^{\P}[\mu\left(s, X_s; L_s, U_s \right)\,|\, \F_u ]\dd s + \sum_{u < s \leq t} \E^{\P}[\gamma\left(s-, X_{s-}; L_{s-}, U_{s-}\right)\D J_s\,|\, \F_u ] \notag \\
& = X_u,
\end{align}
for all $u\leq t\in\T$. If $(X_t)_{t\in\T}$ is continuous, then $\gamma=0$ and $\mu=0$ must hold, and since $\sigma(\tau)=0$ by Definition \ref{maindeftwo}, we must have $(X_t)_{t\in\T} = (Y_t)_{t\in\T}$. If $(X_t)_{t\in\T}$ is a pure-jump process, then $\mu=\sigma=0$, By using the independence of $(J_t)_{t\in\T}$, the following must hold:
\begin{align}
\sum_{u < s \leq t} \E^{\P}[\gamma\left(s-, X_{s-}; L_{s-}, U_{s-}\right)\D J_s\,|\, \F_u ] = \sum_{u < s \leq t} \E^{\P}[\gamma\left(s-, X_{s-}; L_{s-}, U_{s-}\right)\,|\, \F_u ]\E^{\P}[\D J_s\,|\, \F_u ] = 0  \notag.
\end{align}
Since this holds for every $u < s \leq t\in\T$ and $\E^{\P}[\D J_s\,|\, \F_u ] > 0$, $\E^{\P}[\gamma\left(s-, X_{s-}; L_{s-}, U_{s-}\right)\,|\, \F_u ]=0$ must hold for every $u < s \leq t\in\T$. If $\tau\in\T$, that is, the captive process jumped onto a boundary during $\T$, and if $\rho\in\T$ is any time strictly \emph{after} $\tau$ at which $(J_t)_{t\in\T}$ jumps (this always holds due to the finite-activity property), we have:
\begin{equation}
\label{jumpconditionalatboundary}
\E^{\P}[\gamma\left(\rho-, X_{\rho-}; L_{\rho-}, U_{\rho-}\right)\,|\, \F_\tau ]=0 \Rightarrow
\begin{cases}
\E^{\P}[\gamma\left(\rho-, L_{\rho-}; L_{\rho-}, U_{\rho-}\right)\,|\, \F_\tau ]=0 & \text{if $X_{\tau}=L_{\tau}$},\\
\\
\E^{\P}[\gamma\left(\rho-, U_{\rho-}; L_{\rho-}, U_{\rho-}\right)\,|\, \F_\tau ]=0 & \text{if $X_{\tau}=U_{\tau}$}
\end{cases}
\end{equation}
for $\tau < \rho \in\T$. Using Property 3., in Definition \ref{maindeftwo}, we have one of the following cases:
\begin{align}
& \gamma\left(\rho-, L_{\rho-}; L_{\rho-}, U_{\rho-}\right)\geq 0, \notag \\ 
& \gamma\left(\rho-, U_{\rho-}; L_{\rho-}, U_{\rho-}\right)  \leq 0.  \notag
\end{align}
In any of these two cases, Eq. (\ref{jumpconditionalatboundary}) holds only if $\gamma\left(\rho-, X_{\rho-}; L_{\rho-}, U_{\rho-}\right)=0$ for any $\rho > \tau$. Thus, $(X_t)_{t\in\T} = (Y_t)_{t\in\T}$ must hold.
\end{proof}
We now ask: Is there a family of transformations under which a captive jump-diffusion is mapped to another captive jump-diffusion? The answer is yes. First, let $\mathcal{C}_{b}^2(\R) \subset \mathcal{C}(\R)$ be the subspace of continuous locally bounded functions that are twice-differentiable with continuous locally bounded derivatives. From this point onwards, $(X_t)_{t\in\T}$ is a Markovian captive jump-diffusion process as in Remark \ref{diffusionrem}. We shall recall that the jump-times of $(J_t)_{t\in\T}$ are mutually independent of $(X_t)_{t\in\T}$.
\begin{prop}
Let $L,U\in\G$, $f\in\mathcal{C}_{b}^2(\R)$ and $Y_t = f(X_t)$ for all $t\in\T$. If $f$ is monotonic over the domain of $(X_t)_{t\in\T}$, then $(Y_t)_{t\in\T}$ is also a captive jump-diffusion process.
\end{prop}
\begin{proof}
First of all, since $f$ is monotonic, either $f(L_t)\leq Y_t \leq f(U_t)$ if $f$ is increasing or $f(U_t)\leq Y_t \leq f(L_t)$ if $f$ is decreasing, for all $t\in\T$, $\P$-a.s. We write $f(L_t)=\alpha_t$ and $f(U_t)=\beta_t$ for $t\in\T$. Note that $\alpha,\beta\in\G$.
Next, we derive the SDE for $(Y_t)_{t\in\T}$, and check the conditions in Definition \ref{maindeftwo} at these boundaries. 
Using It\^o's lemma, we have the following:
\begin{align}
Y_t &= Y_0 + \int_0^t \frac{\partial f}{\partial x}\d X^c_s + \frac{1}{2}\int_0^t \frac{\partial^2 f}{\partial x^2}\d \left\langle X^c_s, X^c_s\right\rangle + \sum_{0\leq s \leq t}\left[f(X_s) - f(X_{s-})\right]  \notag  \\
&= Y_0 + \int_0^t \left(\frac{\partial f}{\partial x}\mu\left(s, X_s; L_s, U_s \right) + \frac{1}{2}\frac{\partial^2 f}{\partial x^2}\sigma^2\left(s, X_s; L_s, U_s \right) \right)\d s + \int_0^t \frac{\partial f}{\partial x}\sigma\left(s, X_s; L_s, U_s \right)\d W_s \label{secondline} \\
&+ \sum_{0\leq s \leq t}\left[f(X_s) - f(X_{s-})\right]  \notag \\
&\triangleq Y_0 + \int_0^t \hat{\mu}\left(s, Y_s; \alpha_s, \beta_s \right)\d s + \int_0^t \hat{\sigma}\left(s, Y_s; \alpha_s, \beta_s \right)\d W_s + \sum_{0\leq s \leq t}\D Y_s,  \label{thirdline}
\end{align}
for all $t\in\T$. We can write $\hat{\mu}$ and $\hat{\sigma}$ in (\ref{thirdline}) in terms of $Y$, $\alpha$ and $\beta$, since $f$ is monotonic and thus has an inverse, so that we can find some $\hat{\mu}$ and $\hat{\sigma}$ that can provide (\ref{secondline}).
Also, since $f\in\mathcal{C}_{b}^2(\R)$, $\hat{\mu}$ and $\hat{\sigma}$ are locally bounded and continuous, and due to this continuity, $\D Y_t\neq 0$ if and only if $\D J_t=1$ for all $t\in\T$. Since the jump-times of $(J_t)_{t\in\T}$ are mutually independent of $(X_t)_{t\in\T}$, $f$ does not change the distribution of the jumps-times of $(Y_t)_{t\in\T}$; it only acts on the jump-sizes. Hence, there exists some function $\hat{\gamma}$ where we can write $\D Y_s = \hat{\gamma}\left(s-, Y_{s-}; \alpha_{s-}, \beta_{s-}\right)\D J_s$. Therefore, we recover (\ref{mainsdeonemain}):
\begin{align}
Y_t = Y_0 + \int_0^t \hat{\mu}\left(s, Y_s; \alpha_s, \beta_s \right)\d s + \int_0^t \hat{\sigma}\left(s, Y_s; \alpha_s, \beta_s \right)\d W_s + \sum_{0\leq s \leq t}\hat{\gamma}\left(s-, Y_{s-}; \alpha_{s-}, \beta_{s-}\right)\D J_s,  \notag
\end{align}
for all $t\in\T$. We now need to check whether $\hat{\mu}$ and $\hat{\sigma}$ satisfy property 1., and 2., in Definition \ref{maindeftwo}, respectively, and whether $\hat{\gamma}$ satisfies property 3., in Definition \ref{maindeftwo} at the boundaries.
We begin with the case where $f$ is increasing. Then $\alpha < \beta$, and $Y_t$ attains its minimum at $\alpha_t$ when $X_t = L_t$. Since $\partial f/\partial x > 0$ for any $x$, we have
\begin{align}
\hat{\mu}\left(t, \alpha_{t-}; \alpha_t, \beta_t \right) = \frac{\partial f}{\partial L}\mu\left(t, L_t; L_t, U_t \right) \geq \frac{\partial f}{\partial L} \dd L_+(t)/\dd t = \dd \alpha_+(t)/\dd t \notag
\end{align}
given that we know $(\partial^2 f/\partial L^2)\sigma^2\left(s, L_s; L_s, U_s \right) = 0$ since $\sigma\left(s, L_s; L_s, U_s \right) = 0$ by property 2., in Definition \ref{maindeftwo}.
On the other hand, $Y_t$ attains its maximum at $\beta_t$ when $X_t = U_t$, and hence,
\begin{align}
\hat{\mu}\left(t, \beta_{t-}; \alpha_t, \beta_t \right) = \frac{\partial f}{\partial U}\mu\left(t, U_t; L_t, U_t \right) \leq \frac{\partial f}{\partial U} \dd U_+(t)/\dd t = \dd \beta_+(t)/\dd t. \notag
\end{align}
Therefore, $\hat{\mu}$ satisfies property 1., in Definition \ref{maindeftwo}. As for $\hat{\sigma}$ it satisfies 2., in Definition \ref{maindeftwo}, since 
\begin{align}
\hat{\sigma}\left(s, \alpha_s; \alpha_s, \beta_s \right) = \frac{\partial f}{\partial L}\sigma\left(s, L_s; L_s, U_s \right) = 0  \hspace{0.1in} \text{and} \hspace{0.1in}
\hat{\sigma}\left(s, \beta_s; \alpha_s, \beta_s \right) = \frac{\partial f}{\partial U}\sigma\left(s, U_s; L_s, U_s \right) = 0. \notag
\end{align}
Finally, $\alpha_{t-} - Y_{t-} \leq \hat{\gamma}\left(t-, Y_{t-}; \alpha_{t-}, \beta_{t-}\right) \leq \beta_{t-} - Y_{t-}$ must hold $\P$-a.s. since we know that $\alpha_t\leq Y_t \leq \beta_t$ holds and $\D J_t\in(0,1)$ for all $t\in\T$, $\P$-a.s. The case when $f$ is decreasing follows similarly. Here, $\beta < \alpha$, and $Y_t$ attains its minimum at $\beta_t$ when $X_t = U_t$. Hence, having $\partial f/\partial x < 0$ for any $x$,
\begin{align}
\hat{\mu}\left(t, \beta_{t-}; \alpha_t, \beta_t \right) = \frac{\partial f}{\partial U}\mu\left(t, U_t; L_t, U_t \right) \geq \frac{\partial f}{\partial U} \dd U_+(t)/\dd t = \dd \beta_+(t)/\dd t \notag
\end{align}
Also, since $Y_t$ attains its maximum at $\alpha_t$ when $X_t = L_t$, we have the following:
\begin{align}
\hat{\mu}\left(t, \alpha_{t-}; \alpha_t, \beta_t \right) = \frac{\partial f}{\partial L}\mu\left(t, L_t; L_t, U_t \right) \leq \frac{\partial f}{\partial L} \dd L_+(t)/\dd t = \dd \alpha_+(t)/\dd t. \notag
\end{align}
Again, $\hat{\mu}$ satisfies property 1., in Definition \ref{maindeftwo}. We already know $\hat{\sigma}$ satisfies property 2., in Definition \ref{maindeftwo}. Finally, since $\beta_t\leq Y_t \leq \alpha_t$ for all $t\in\T$, $\P$-a.s., we must have $\beta_{t-} - Y_{t-} \leq \hat{\gamma}\left(t-, Y_{t-}; \alpha_{t-}, \beta_{t-}\right) \leq \alpha_{t-} - Y_{t-}$ for all $t\in\T$, $\P$-a.s., which is property 3., in Definition \ref{maindeftwo}.
\end{proof}
\begin{ex}
Let $Y_t = e^{X_t}$ for all $t\in\T$. Then, $(Y_t)_{t\in\T}$ is a captive jump-diffusion process, where 
\begin{align}
e^{L_t} \leq Y_t \leq e^{U_t} \hspace{0.15in} \text{for all $t\in\T$, $\P$-a.s.} \notag
\end{align}
\end{ex}
\begin{ex}
Let $L_t > 0$ and $Y_t = X_t^{-1}$ for all $t\in\T$. Then, $(Y_t)_{t\in\T}$ is a captive jump-diffusion process, where 
\begin{align}
U_t^{-1} \leq Y_t \leq L_t^{-1} \hspace{0.15in} \text{for all $t\in\T$, $\P$-a.s.} \notag 
\end{align}
\end{ex}
\subsection{Path-dependent captive jump processes}
So far, the generated jump processes have been restricted to never leave a single confined space. In Meng\"ut\"urk \& Meng\"ut\"urk (2021b) captive diffusion processes have been introduced, which are confined to stay within a finite space for some time (in their terminology a ``tunnel'') before being allowed to diffuse into another confined space in a following period of time. This behaviour makes of a captive diffusion a path-dependent process, because of the implicit monitoring that is required in order to control the shift of the diffusion from one corridor to another. In what follows, we show how captive jump processes can be constructed such that for part of the time they stay confined in a bounded space before they are allowed to {\it jump} to another restricted area for some time. This behaviour requires the modelling of controlled jumps (arrival times and size) and continuous monitoring of the process, which in turn makes the captive jump process path-dependent. As the simulation in Section \ref{SecApplications} will illustrate, the restricted areas (``tunnels'') captive jump processes may jump to are not required to be adjacent, in the sense that the process may skip neighbouring corridors. Such features render captive jump processes flexible and useful in a variety of applications.

Accordingly, we go on to develop a way to extend Definition \ref{maindeftwo} by allowing \textit{internal} boundaries to appear. In doing so, we follow \cite{52}, and introduce multiple time-segments denoted by $\T^{(j)}\subseteq\T$ for $j=0,\ldots,m\in\mathbb{N}_+$, for a fixed $m\geq 1$. Here,
$
\T^{(j)}=[\tau_{\text{start}}^{(j)},\tau_{\text{end}}^{(j)}]
$
such that $0\leq \tau_{\text{start}}^{(j)} < \tau_{\text{end}}^{(j)} \leq T$ for $j=0,\ldots,m$. 
We choose $\T^{(0)}=\T^{(m)}=\T$, that is, $\tau_{\text{start}}^{(0)}=\tau_{\text{start}}^{(m)}=0$ and $\tau_{\text{end}}^{(0)}=\tau_{\text{end}}^{(m)}=T$. Each of these time-segments govern a boundary process $(g_t^{(j)})_{t\in\T^{(j)}}$, where we have $g^{(0)}\in\widetilde{\G}^{(l)}$, $g^{(m)}\in\widetilde{\G}^{(u)}$ and $g^{(j)}\in\G$ for any $j\neq 0$ and $j \neq m$ given that $m>1$. We emphasize that $(g_t^{(0)})_{t\in\T^{(0)}}$ stands for $(L_t)_{t\in\T}$ and $(g_t^{(m)})_{t\in\T^{(m)}}$ stands for $(U_t)_{t\in\T}$ in the previous section. These are the master, or outer, boundaries.
We collect all boundaries by 
\begin{align}
\textbf{g}_t=\{g_t^{(0)}, \ldots, g_t^{(m)}\}. \notag
\end{align}
In the situation where there is a $g_t^{(j)}$, for $0<j<m$, that is not defined over some $t\in\T$, then $\textbf{g}_t$ does not include such a $g_t^{(j)}$ at the time $t\in\T$. Finally, for any $j < k$ at any $t\in\T^{(j)}\cap\T^{(k)}\neq\emptyset$, we require the ordering $g_t^{(j)} < g_t^{(k)}$, and for any pair $\{g_t^{(j)}, g_t^{(k)}\}$ where $\T^{(j)}\cap\T^{(k)}=\emptyset$, their order can be pairwise arbitrary. 
\begin{rem}
The \emph{master boundaries} $g^{(0)}$ and $g^{(m)}$ define the largest bounded domain on which captive processes may evolve, and each $g^{(j)}$ is called an \emph{internal boundary}, for $j\neq 0$, $j\neq m$ and $m>1$.
\end{rem}
Next, we introduce progressively-measurable and increasing processes $(\tau^{(j)}_t)_{t\in\T^{(j)}}$ given by
\begin{align}
\tau^{(j)}_t = \tau_{\text{start}}^{(j)} \vee \sup(s: \D X_s \neq 0\hspace{0.1in} \text{for} \hspace{0.1in} \tau_{\text{start}}^{(j)} \leq s\leq t\in\T^{(j)} ), \notag 
\end{align}
for $j=0,\ldots,m$, where we adopt the convention $\sup\emptyset = -\infty$. Hence, if there is no jump in a given time period $\T^{(j)}$, then $\tau^{(j)}_t = \tau_{\text{start}}^{(j)}$ for all $t$ in the period $\T^{(j)}$. We also introduce a monitoring process $(\Psi_t)_{t\in\T}$ that records the values of $(X_t)_{t\in\T}$ at $(\tau^{(j)}_t)_{t\in\T^{(j)}}$ for $j=0,\ldots,m$. As such, we let $(\Psi_t)_{t\in\T}$ be the non-anticipative and set-valued process 
\begin{align}
\Psi_t=((X_{\tau^{(j)}_t}, \tau^{(j)}_t) : \, \tau^{(j)}_t \leq t, \,\, j=0,\ldots,m),  \notag 
\end{align}
for all $t\in\T$. In the case we require a function to be continuous with respect to $\Psi$, the topology we mean is continuity with respect to the elements of $\Psi$. The following definition may also be viewed as a lemma to Proposition \ref{quasicaptive} below.
\begin{defn}
\label{maindeftwoextended}
A \emph{piecewise-confined captive jump process} $(X_t)_{t\in\T}\in\mathcal{C}(\R)$ is the solution to the stochastic differential equation
\begin{align}
\label{mainsdeonemainextended}
X_t = x_0 + \int_0^t \mu\left(s, \Psi_s, X_s; \textbf{g}_s\right)\dd s + \int_0^t \sigma\left(s, \Psi_s, X_s; \textbf{g}_s\right)\dd M_s + \sum_{0\leq s \leq t} \gamma\left(s-, \Psi_{s-}, X_{s-}; \textbf{g}_{s-}\right)\D J_s,
\end{align}
where $X_0 = x_0 \in [g^{(0)}_0,g^{(m)}_0)$. Here, $\mu$ is a locally bounded map that is continuous (possibly except at $\D X\neq 0$), $\sigma$ is a locally bounded continuous map, and $\gamma$ is a locally bounded \cadlag map, satisfying
\begin{enumerate}
\item $\mu\left(t,\Psi_{t}, g^{(j)}_{t-}; \textbf{g}_t\right)\geq \dd g^{(j)}_+(t)/\dd t + \D g^{(j)}_t$, if $X_{\tau^{(j)}_t} \geq g^{(j)}_{\tau^{(j)}_t}$ for all $t\in\T^{(j)}$;
\item $\mu\left(t,\Psi_{t}, g^{(j)}_{t-}; \textbf{g}_t\right)\leq \dd g^{(j)}_+(t)/\dd t + \D g^{(j)}_t$, if $X_{\tau^{(j)}_t} < g^{(j)}_{\tau^{(j)}_t}$ for all $t\in\T^{(j)}$;
\item $\sigma\left(t,\Psi_{t},g^{(j)}_{t-}; \textbf{g}_t\right) = 0$ for all $t\in\T^{(j)}$;
\item $g^{(0)}_{t-} - X_{t-} \leq \gamma\left(t-, \Psi_{t-}, X_{t-}; \textbf{g}_{t-}\right) \leq g^{(m)}_{t-} - X_{t-}$ for all $t\in\T$,
\end{enumerate}
for $j=0,\ldots,m$ $\P$-a.s., given that $(M_t)_{t\in\T}\in\mathcal{M}(\R)$ and $(J_t)_{t\in\T}\in\mathcal{J}(\R)$ are mutually independent.
\end{defn}
There are a few observations worth making at this stage. First, we notice that $\D g^{(j)}_t = 0$ for any $j\neq 0$ and $j\neq m$, since they belong to $\mathcal{G}$---i.e., internal boundaries cannot have discontinuities. This however is no real restriction, since multiple internal confining functions $(g^{(j)}_t)_{t\in\T}$ can be used in sequence to construct a piecewise process that behaves like an internal corridor with jumps. The reason for requiring each individual $g^{(j)}\in\mathcal{G}$ for any $j\neq 0$ and $j\neq m$ is a technical requirement that is needed for Proposition \ref{quasicaptive} below to ensure captivity in \emph{each} corridor segment that would otherwise be broken over time periods where $(X_t)_{t\in\T}$ is continuous. Second, the SDE coefficients are path-dependent, that is, they monitor past values of the captive jump process at $(\tau^{(j)}_t)_{t\in\T^{(j)}}$. This means that such processes are non-Markovian even if $(M_t)_{t\in\T}$ and $(J_t)_{t\in\T}$ are Markov processes.
Third, the initial condition $x_0 \in [g^{(0)}_0,g^{(m)}_0)$ does not include $g^{(m)}_0$, since we associated the case $X(\tau_{\text{start}}^{(j)}) = g^{(j)}(\tau_{\text{start}}^{(j)})$ with Property 1., above, which could have been associated with Property 2. having initial condition $x_0 \in (g^{(0)}_0,g^{(m)}_0]$. 
Fourth, Property 4. only requires the jump coefficient $\gamma$ to be bounded with respect to the \emph{master} boundaries $\{g^{(0)}_t,g^{(m)}_t\}$. Hence, Definition \ref{maindeftwoextended} allows for jump-sizes to exceed \emph{internal} corridors---this will be useful for applications of captive jump processes. 
\begin{prop}
\label{quasicaptive}
The following statements hold $\P$-almost surely.
\begin{enumerate}
\item For any $t\in\T$, $g_t^{(0)} \leq X_t \leq g_t^{(m)}$. 
\item For any $j\neq 0$ and $j\neq m$, if $X_{\tau_{\text{start}}^{(j)}} \geq g^{(j)}_{\tau_{\text{start}}^{(j)}}$ then $X_{t} \geq g^{(j)}_{t}$ if $\D J_t=0$ for all $t\in\T^{(j)}$. If $\D J_t=1$ for some $t\in\T^{(j)}$ then $\P(X_{t} < g^{(j)}_{t})  \geq 0$.
\item For any $j\neq 0$ and $j\neq m$, if $X_{\tau_{\text{start}}^{(j)}} < g^{(j)}_{\tau_{\text{start}}^{(j)}}$ then $X_{t} \leq g^{(j)}_{t}$if $\D J_t=0$ for all $t\in\T^{(j)}/(\tau_{\text{start}}^{(j)})$. If $\D J_t=1$ for some $t\in\T^{(j)}$ then $\P(X_{t} > g^{(j)}_{t})  \geq 0$.
\end{enumerate}
\end{prop}
\begin{proof}
This follows from our construction of Definition \ref{maindeftwoextended} together with Proposition \ref{captivepropmain}, and Proposition 2.8 in \cite{52}, combined with Property 4., in Definition \ref{maindeftwoextended}, that allows $\D X_t$ for $t\in\T$ to be large enough for $(X_t)_{t\in\T}$ to exceed the boundary $g^{(j)}$ for $j\neq 0$ and $j\neq m$, but not from the master boundaries $\{g^{(0)},g^{(m)}\}$, $\P$-a.s.
\end{proof}
Proposition \ref{quasicaptive} also tells us that $\gamma$ can be chosen in such a way that if $(X_t)_{t\in\T}$ is in an internal corridor, then it may remain in that corridor even if there is a jump---Property 4. of Definition \ref{maindeftwoextended} allows for this feature. 
However, we are more interested in modelling systems where a transition between two internal corridors is allowed if there is a jump, and \emph{only} if there is a jump, as Definition \ref{maindeftwoextended} offers.

\section{Applications}\label{SecApplications}
We shall work with mean-reverting captive jump processes, see Example \ref{meanrevertexamplejump}, for demonstration. We refer the reader to \cite{50} for other families of captive processes (with no jumps), which can be extended by use of Definition \ref{maindeftwo}.

\subsection{Confinement within circular domains}
We first consider a setup without internal corridors. Let $(X^{(1)}_t)_{t\in\T}$ be a captive jump-diffusion, where $X^{(0)}_0\in[a,d]$ for some $0 \leq a < d  < \infty$, and let $(X^{(2)}_t)_{t\in\T}$ be another captive jump-diffusion where $X^{(0)}_0\in[0,2\pi]$. These two processes are governed by
\begin{align}
\d X^{(i)}_t = (\beta^{(i)} - X^{(i)}_t ) \dd t + (X^{(i)}_t - L^{(i)})\left(U^{(i)} - X^{(i)}_t\right)\dd W^{(i)}_t +\theta^{(i)}_{t-} \min\left(X^{(i)}_{t-} - L^{(i)}, U^{(i)} - X^{(i)}_{t-}\right)\D J^{(i)}_t.  \notag
\end{align}
Here, $(W^{(1)}_t)_{t\in\T}$ and $(W^{(2)}_t)_{t\in\T}$ may be correlated. Similarly, $(J^{(1)}_t)_{t\in\T}$ and $(J^{(2)}_t)_{t\in\T}$ may also be correlated. The process $(\theta_t)_{t\in\T}$ may be any \cadlag map as long as $\theta_t\in[-1,1]/\{0\}$ for all $t\in\T$. 

For the examples below, we simulate in advance a random path for $(\theta_t)_{t\in\T}$ by uniformly sampling its values at every step on a discrete time grid. The boundaries are given by
\begin{align}
&a = L^{(1)} < \beta^{(1)} < U^{(1)} = d, \notag \\
&0 = L^{(2)} < \beta^{(2)} < U^{(2)} = 2\pi.  \notag
\end{align}
Next we construct the two-dimesional process $(P_t)_{t\in\T}$ given by
$(X^{(1)}_t,X^{(2)}_t)_{t\in\T}$
on a polar coordinate system, where $(X^{(1)}_t)_{t\in\T}$ models the distance from the origin and $(X^{(2)}_t)_{t\in\T}$ is the radian process. If $r\geq d$ is the radius of a circle, we can choose $a$ and $d$ such that $(P_t)_{t\in\T}$ evolves either inside the outer circular domain or within a circular corridor inside the domain. First we show the case where an outer circle serves as the outer boundary for the confined circular domain.
\begin{figure}[h!] 
\begin{center}
{\includegraphics[scale=.225]{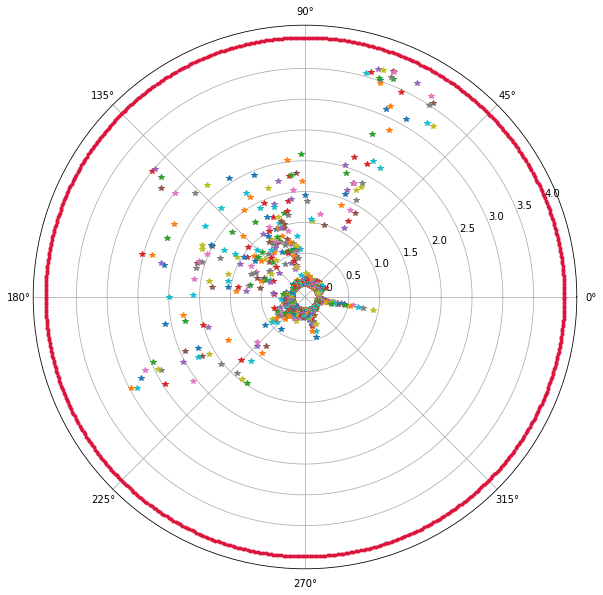}} 
\hspace{1.5cm}
{\includegraphics[scale=.225]{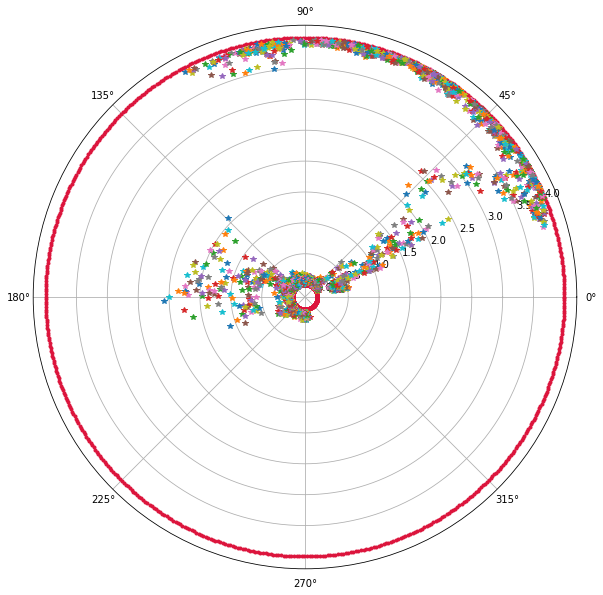}} 
\caption[Captiveone]{Here, $a=0$, $d=r=4$ and $\beta = \frac{1}{2}(a + d)$}
\end{center}
\end{figure}
We can see two different \emph{accumulation} behaviours in Figure 2. On the left-hand side, we see that the process tends to evolve towards the origin of the circle, while on the right-hand side, the process visits the border of the circle many times during the simulation period. 
\begin{rem}
\label{spherecaptive}
One can introduce a third captive jump-diffusion process $(X^{(3)}_t)_{t\in\T}$ as the second orthogonal radian coordinate, and project $(X^{(1)}_t,X^{(2)}_t,X^{(3)}_t)_{t\in\T}$ inside a sphere. This indicates how the construction of captive jump processes can be extended to processes taking values in confined domains in $\R^n$, $n\in\N$.
\end{rem}
In Figure 3, we shrink the domain inside the circle to keep the paths within shrinking rings, each shaped as a toroid. Again, this can be extended to three dimensions in the spirit of Remark \ref{spherecaptive}.
\begin{figure}[h!] 
\begin{center}
{\includegraphics[scale=.225]{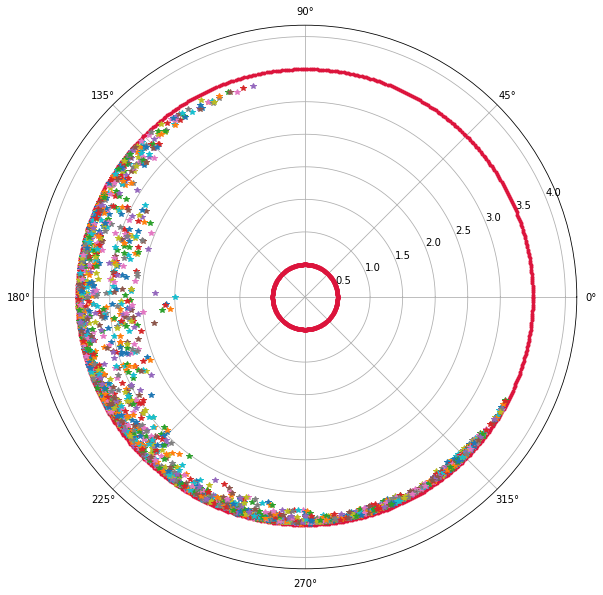}} 
\hspace{1.5cm}
{\includegraphics[scale=.225]{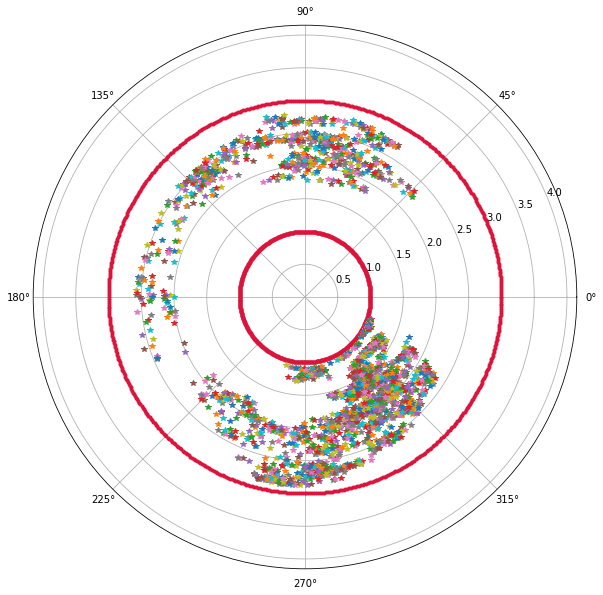}} 
\\\vspace{.5cm}
{\includegraphics[scale=.225]{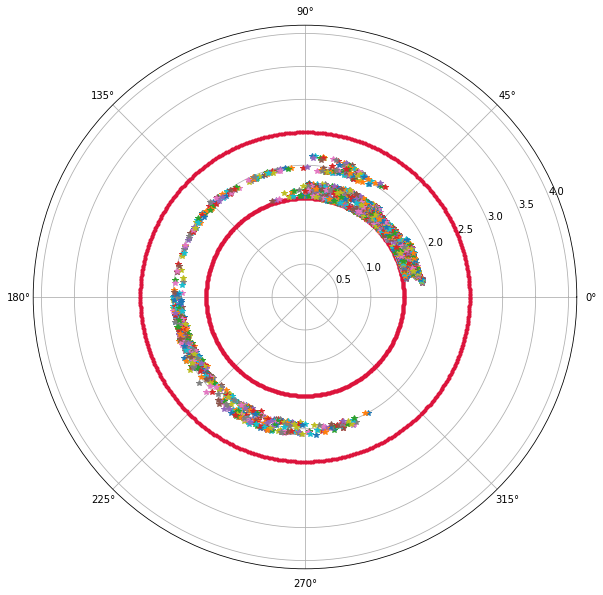}} 
\hspace{1.5cm}
{\includegraphics[scale=.225]{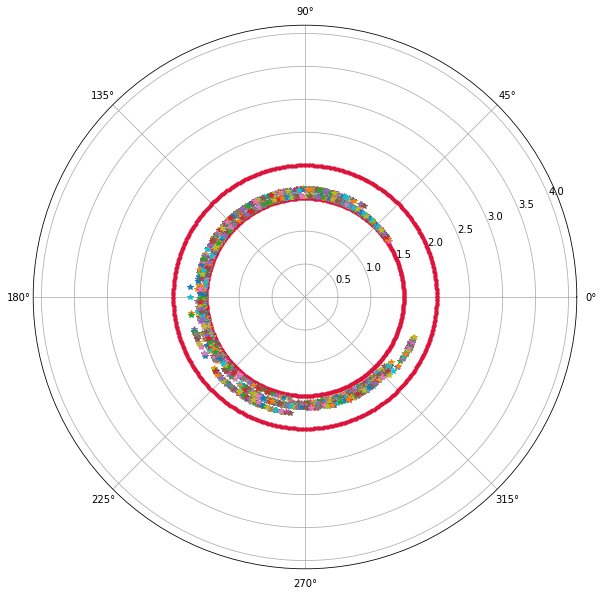}} 
\caption[Captiveone]{Confining boundaries \{a,d\}: top-left (0,3.5), top-right (1,3), bottom-left: (1.5,2.5), bottom-right: (1.5,2).}
\end{center}
\end{figure}
\\
This setup can be used to model systems with a central gravitational force which keeps stochastic particles within circular corridors. Here, even if $(X^{(1)}_t,X^{(2)}_t)_{t\in\T}$ jumps, it can never cross the master (outer) boundaries and so break free from the confined toroidal space. In the next section, we shall add internal corridors, which the particles are allowed to trespass only if they jump far enough.

\subsection{Captive jumps across circular domains}
Now we consider the situation where, in addition to the outer boundaries, there are two inner boundaries for the distance-process $(X^{(1)}_t)_{t\in\T}$; we keep the radian-process $(X^{(2)}_t)_{t\in\T}$ as in the previous section. We set $\T^{(0)}=\T^{(1)}=\T^{(2)}=\T^{(3)}=\T$ such that $(g^{(0)}_t)_{t\in\T}=a$, $(g^{(1)}_t)_{t\in\T}=b$, $(g^{(2)}_t)_{t\in\T}=c$ and $(g^{(3)}_t)_{t\in\T}=d$. Hence, $\tau_{\text{start}}^{(j)}=0$ and $\tau_{\text{end}}^{(j)}=T$ for $j=1,\ldots,3$. So, we simply write $\tau^{(j)}_t = \tau_t$ for $j=1,\ldots,3$ and $t\in\T$.
We initialize $(X^{(1)}_t)_{t\in\T}$ such that it starts within the inner circular corridor, where $X^{(0)}_0\in[a,b]$ for some $0 \leq a < b < c < d < \infty$, where $[a,b]$ forms the lowermost corridor, $[b,c]$ forms a mid-corridor and $[c,d]$ forms the uppermost corridor for $(X^{(1)}_t)_{t\in\T}$, which is now governed by
\begin{align}
\d X^{(1)}_t = (\beta^{(1)}_t(\Psi^{(1)}_t) - X^{(1)}_t ) \dd t + \left(\prod_{g^{(j)}\in\boldsymbol{g}}(X^{(1)}_t - g^{(j)})\right)\dd W^{(1)}_t +\theta^{(1)}_{t-} \min\left(X^{(1)}_{t-} - a, d - X^{(1)}_{t-}\right)\D J^{(1)}_t.  \notag
\end{align}
Here, $\boldsymbol{g}=(a,b,c,d)$ and $(\beta^{(1)}_t(\Psi^{(1)}_t))_{t\in\T}$ is given by
\begin{equation*}
\label{betalogic}
\beta_t(\Psi_{t}) =
\begin{cases}
w_1*a+ (1-w_1)*b & \text{if $X_{\tau_t} \in[a,b)$},\\
w_2*b+ (1-w_2)*c & \text{if $X_{\tau_t} \in[b,c)$},\\
w_3*c+ (1-w_3)*d & \text{if $X_{\tau_t} \in[c,d]$},\\
\end{cases}
\end{equation*}
for some $w_1,w_2,w_3\in(0,1)$. Next, we construct a process $(P_t)_{t\in\T}$ by introducing the two-dimesional process $(X^{(1)}_t,X^{(2)}_t)_{t\in\T}$ taking values in the polar coordinate system. The plots in Figure 4 show samples of $(X^{(1)}_t)_{t\in\T}$ and the associated process $(X^{(1)}_t,X^{(2)}_t)_{t\in\T}$.
\begin{figure}[h!] 
\begin{center}
\begin{align*}
&{\includegraphics[scale=.25]{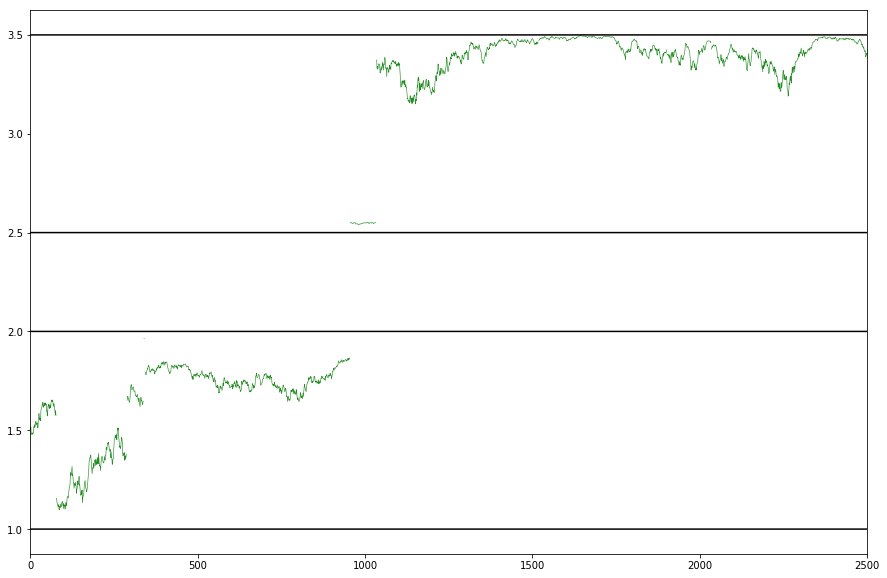}}&
&{\includegraphics[scale=.25]{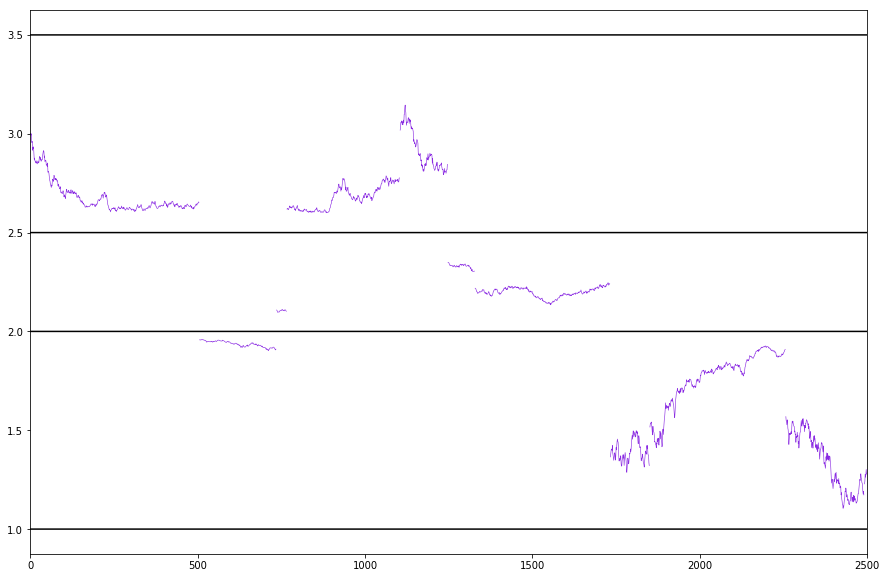}}& 
\end{align*}
\\
\begin{align*}
&{\includegraphics[scale=0.25]{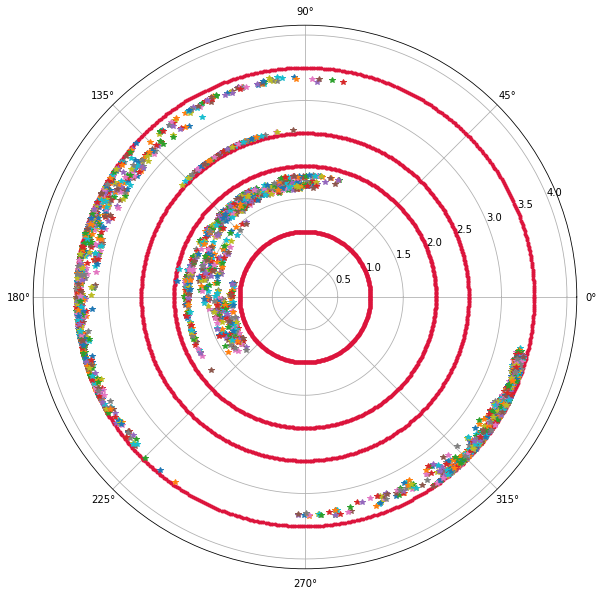}}&
&{\includegraphics[scale=.25]{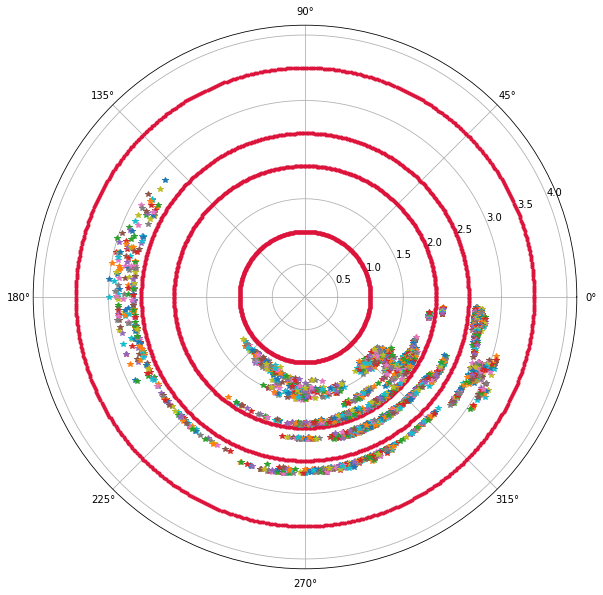}}& 
\end{align*}
\end{center}
\caption[Captiveone]{Here, $a=1$, $b=2$, $c=2.5$, $d=3.5$. Also, $w_1=w_2=w_3=0.5$}
\end{figure}
The simulation on the left-hand side shows how $(X^{(1)}_t,X^{(2)}_t)_{t\in\T}$ in polar coordinates remains within the lower or upper corridors, where the transition from $[a,b)$ to $[c,d]$ occurs when $(X^{(1)}_t)_{t\in\T}$ jumps far enough to skip the mid-corridor $[b,c)$. On the right-hand side, $(X^{(1)}_t,X^{(2)}_t)_{t\in\T}$ visits every corridor, depending on the size of the jumps. Of course, the process might jump within a specific corridor without necessarily leaving it. 

The conditional state probabilities can be calculated, where in our context, \emph{state} means a corridor. That is, it is possible to calculate the probability of the captive jump process to move from one corridor to another. We express the conditional state probability by 
\begin{align}
\P_t([k,l),[a,b)):= \P\left(X^{(1)}_{t}\in[k,l) \,|\, \F_{t-}, X^{(1)}_{t-}\in[a,b) \right), \notag
\end{align}
for $[k,l)\in([b,c), [c,d))$.
We also define the set
\begin{align}
\mathcal{S}^{[k,l)}:=\left[ \frac{(k - X^{(1)}_{t-})}{\min\left(X^{(1)}_{t-} - a, d - X^{(1)}_{t-}\right)} \,,\, \frac{(l - X^{(1)}_{t-})}{\min\left(X^{(1)}_{t-} - a, d - X^{(1)}_{t-}\right)}   \right). \notag
\end{align}
Finally, we ask for $(\theta^{(1)}_t)_{t\in\T}$ and $(J^{(1)}_t)_{t\in\T}$ to be mutually exclusive. 
Then, we have the following:
\begin{align}
\P_t([k,l),[a,b)) & = \P\left(\left.\theta_{t}\in \mathcal{S}^{[k,l)} \,\right|\, \F_{t-}, X^{(1)}_{t-}\in[a,b) \right)\P\left(\D J_{t} =1 \,|\, \F_{t-}, X^{(1)}_{t-}\in[a,b) \right),  \notag
\end{align}
for $[k,l)\in([b,c), [c,d))$. Note that since in this model $\theta_t\in[-1,1]/(0)$ for all $t\in\T$, we have
\begin{align}
\P_t([k,l),[a,b)) = 0 \hspace{0.15in} \text{if} \hspace{0.15in} (k - X^{(1)}_{t-}) > \min\left(X^{(1)}_{t-} - a, d - X^{(1)}_{t-}\right). \notag
\end{align}
This shows that $(X^{(1)}_t)_{t\in\T}$ has a higher probability of moving to another corridor if it is closer to a boundary of that corridor. These probabilities can be calculated from any one corridor to another, i.e., we also have
\begin{align}
\P_t([k,l),[c,d)) & = \P\left(\left.\theta_{t}\in \mathcal{S}^{[k,l)} \,\right|\, \F_{t-}, X^{(1)}_{t-}\in[c,d) \right)\P\left(\D J_{t} =1 \,|\, \F_{t-}, X^{(1)}_{t-}\in[c,d) \right) \notag
\end{align}
for $[k,l)\in([a,b), [b,c))$. Again, since $\theta_t\in[-1,1]/(0)$ for all $t\in\T$, we have
\begin{align}
\P_t([k,l),[c,d)) = 0 \hspace{0.15in} \text{if} \hspace{0.15in} (l - X^{(1)}_{t-}) < - \min\left(X^{(1)}_{t-} - a, d - X^{(1)}_{t-}\right). \notag
\end{align}
Finally, the probability of the captive jump process moving from the mid-corridor to either the inner or outer corridor is given by
\begin{align}
\P_t([k,l),[b,c)) & = \P\left(\left.\theta_{t}\in \mathcal{S}^{[k,l)} \,\right|\, \F_{t-}, X^{(1)}_{t-}\in[b,c) \right)\P\left(\D J_{t} =1 \,|\, \F_{t-}, X^{(1)}_{t-}\in[b,c) \right), \notag
\end{align}
for $[k,l)\in([a,b), [c,d))$. Hence, the set $\mathcal{S}^{[k,l)}$ serves for all possible changes of state in this model. All expressions can be further simplified if $(\theta_t)_{t\in\T}$ and $(J_t)_{t\in\T}$ are mutually independent such that
\begin{align}
\P_t([k,l),[x_1,x_2)) & = \P\left(\left.\theta_{t}\in \mathcal{S}^{[k,l)} \,\right|\, X^{(1)}_{t-}\in[x_1,x_2) \right)\P\left(\D J_{t} =1 \right). \notag
\end{align}
This setup can be used to model a system in which a stochastic particle may jump from one energy state to another with a large-enough jump that is induced by a sufficiently energetic exogenous shock. 
Here one might make a connection to a quantum mechanical system, whereby one is interested in modelling the transition of the stochastic wave function of a quantum particle from one energy state to another. Another application that could be considered is the modelling of quantum tunnelling in a stochastic setting. Here the potential walls (barriers) of the quantum system that the particle ``overcomes'' could be modelled by the boundaries of an internal corridor, which are overcome by a large-enough jump. In this context, while the quantum particle is modelled by a captive jump process able to overcome walls, a classical particle would be modelled by a captive diffusion process trapped within a (possibly time-dependent) corridor, see \cite{52}. The concept of domain boundaries (or walls), which produce clusters unable to disperse (if unaided by external intervention), abounds in many fields of physics, but it is also encountered in finance (e.g., volatility clustering, herding in markets), chemistry, sociology and psychology to just name a few.


\begin{thebibliography}{99}
\bibitem{1aa} It\^o, K., \textit{On Stochastic Differential Equations on a Differentiable Manifold I}, Nagoya Math. J. 1 (1950)
\bibitem{1ab} It\^o, K., \textit{On Stochastic Differential Equations on a Differentiable Manifold I}, Mem. Coll. Sci. Univ. Kyoto. Ser. A 28 (1953)
\bibitem{1} Skorokhod, A. V., \textit{Stochastic Equations for Diffusion Processes in a Bounded Region I}, Theory of Probability and Its Applications 6 (1961)
\bibitem{2}  Skorokhod, A. V., \textit{Stochastic Equations for Diffusion Processes in a Bounded Region II}, Theory of Probability and Its Applications 7 (1962)
\bibitem{3} Dyson, F. J., \textit{A Brownian-Motion Model for the Eigenvalues of a Random Matrix}, Journal of Mathematical Physics 3 (1962)
\bibitem{4} McKean, H.P., \textit{Skorohod's Integral Equation for a Reflecting Barrier Diffusion}, J. Math. Kyoto Univ., 3 (1963)
\bibitem{6} Durrett, R.T., Iglehart, D.L., \textit{Functionals of Brownian Meander and Brownian Excursion}, The Annals of Probability 5 (1977)
\bibitem{8} Pitman, J., and Yor, M., \textit{Bessel Processes and Infinitely Divisible Laws}, in D. Williams (ed.), Stochastic Integrals, 851 Berlin: Springer-Verlag (1980)
\bibitem{9} Harrison, J. M., Reiman, M. I., \textit{On the Distribution of Multidimensional Reflected Brownian Motion}, SIAM Journal on Applied Mathematics 41 (1981)
\bibitem{11} Cox, J.C., Ingersoll, J.E. Jr., Ross, S.A., \textit{A Theory of the Term Structure of Interest Rates}, Econometrica 53 (1985)
\bibitem{14} Asmussen, S., Glynn, P., Pitman, J. \textit{Discretization Error in Simulation of One-Dimensional Reflecting Brownian Motion}, The Annals of Applied Probability 5 (1995)
\bibitem{25} G\"oing-Jaeschke, A., Yor, M., \textit{A Survey and Some Generalizations of Bessel Processes}, Bernoulli 9 (2003)
\bibitem{27a} Debbasch, F., \textit{A Diffusion Process in Curved Space-Time}. Journal of Mathematical Physics (2004)
45 (7) 2744
\bibitem{28} Obloj, J., Yor, M., \textit{An Explicit Skorokhod Embedding for the Age of Brownian Excursions and Az\'ema Martingale}, 
Stochastic Processes and their Applications 110 (2004)
\bibitem{29} Katori, M., Tanemura, H., \textit{Symmetry of Matrix-Valued Stochastic Processes and Non-colliding Diffusion Particle Systems}, Journal of Mathematical Physics 45 (2004)
\bibitem{30} Deuschel, J.-D., Zambotti, L., \textit{Bismut-Elworthy’s Formula and Random Walk Representation for SDEs with Reflection}, Stochastic Processes and their Applications 115 (2005)
\bibitem{30ab} Linetsky, V., \textit{On the Transition Densities for Reflected Diffusions}, Annals of Applied Probability 37 (2005)
\bibitem{33} Pitman, J., Yor, M., \textit{It\^o's Excursion Theory and Its Applications}, Japanese Journal of Mathematics 2 (2007)
\bibitem{41} Yen, J.Y., Yor, M., \textit{Local Times and Excursion Theory for Brownian Motion}, Springer (2013)
\bibitem{44} Katori, M., \textit{Determinantal Martingales and Noncolliding Diffusion Processes}, Stochastic Processes and their Applications 124 (2014)
\bibitem{47} Karl L., Dong W., \textit{Nonintersecting Brownian Motions on the Unit Circle}, The Annals of Probability 44 (2016)
\bibitem{48ab} Pitman, J., Winkel, M., \textit{Squared Bessel Processes of Positive and Negative Dimension Embedded in Brownian Local Times}, Electronic Communications in Probability (2018)
\bibitem{50} Meng\"ut\"urk, L. A., Meng\"ut\"urk, M. C., \textit{Captive Diffusions and Their Applications to Order-Preserving Dynamics}, Proceedings of the Royal Society A: Mathematical, Physical and Engineering Sciences (2020) 
\bibitem{51} Meng\"ut\"urk, L. A., Meng\"ut\"urk, M. C., \textit{Loewner-Captive Hermitian Diffusions and Their Applications in Quadratic Optimization}, Working paper
\bibitem{52} Meng\"ut\"urk, L. A., Meng\"ut\"urk, M. C., \textit{Piecewise-Tunneled Captive Processes and Corridored Random Particle Systems}, Working paper
\end{thebibliography}
\end{document}